\documentclass[11pt]{amsart}

\usepackage[margin=3cm]{geometry}
\usepackage[pdftex]{graphicx}

\usepackage{amsfonts}
\usepackage{amsmath}
\usepackage{amssymb}
\usepackage{amsthm}

\usepackage[T1]{fontenc}

\usepackage{paralist}

\usepackage[colorlinks,cite color=blue,pagebackref=true,pdftex]{hyperref}

\newcommand\Defn[1]{\textbf{\color{black}#1}}

\newcommand\R{\mathbb{R}}               
\newcommand\Z{\mathbb{Z}}

\newtheorem{thm}{Theorem}
\newtheorem{cor}[thm]{Corollary}
\newtheorem{lem}[thm]{Lemma}
\newtheorem{prop}[thm]{Proposition}

\theoremstyle{definition}

\newtheorem*{obs*}{Observation}

\title{Minkowski complexes and convex threshold dimension}

\author{Florian Frick}
\address{Department of Mathematics, %
Cornell University, Ithaca, NY 14853, %
USA}
\email{ff238@cornell.edu}

\author{Raman Sanyal}
\address{Institut f\"ur Mathematik, Goethe-Universit\"at Frankfurt, Germany}
\email{sanyal@math.uni-frankfurt.de}

 %  	Mixed volumes and related topics

\date{\today}

\begin{document}

\begin{abstract}
    For a collection of convex bodies $P_1,\dots,P_n \subset \R^d$ containing
    the origin, a \emph{Minkowski complex} is given by those subsets whose
    Minkowski sum does not contain a fixed basepoint.  Every simplicial
    complex can be realized as a Minkowski complex and for convex bodies on
    the real line, this recovers the class of threshold complexes.  The
    purpose of this note is the study of the \emph{convex threshold dimension}
    of a complex, that is, the smallest dimension in which it can be realized
    as a Minkowski complex.  In particular, we show that the convex threshold
    dimension can be arbitrarily large. This is related to work of Chv\'atal
    and Hammer (1977) regarding forbidden subgraphs of threshold graphs. We
    also show that convexity is crucial this context.
\end{abstract}

\maketitle

\renewcommand\l{\lambda}%

A simplicial complex $\Delta$ on vertices $[n] := \{1,\dots,n\}$ is a
\Defn{threshold complex} if there are real numbers $\l_1,\dots,\l_n, \mu \in
\R$ with $0 \le \l_i \le \mu$ for all $i = 1,\dots,n$ such that for any
$\sigma \subseteq [n]$
%\begin{equation}\label{eqn:thres_weight}
\[
    \sigma \ \in \ \Delta \qquad \text{if and only if} \qquad \sum_{i \in
    \sigma} \l_i \ < \ \mu.
\]
%\end{equation}
Threshold complexes (or hypergraphs) where proposed by Golumbic~\cite{Gol80}
as a higher-dimensional generalization of the threshold \emph{graphs} of
Chv\'atal and Hammer~\cite{CH77}; see also~\cite{RRST}. If we assume that $0
\le \l_1 \le \cdots \le \l_n \le \mu$, then for any $i \in \sigma \in \Delta$
and $j < i$, we have $(\sigma \setminus i) \cup j \in \Delta$. Hence,
threshold complexes are \Defn{shifted} in the sense of Kalai~\cite{kalai} 
and topologically wedges of (not necessarily equidimensional)
spheres. See~\cite{Klivans} and~\cite{edelman} for more information regarding
the combinatorics and topology of threshold and shifted complexes.

\renewcommand\P{\mathcal{P}}%
The purpose of this note is to investigate a generalization of threshold
complexes inspired by convex geometry. For that, let $\P = (P_1, \dots, P_n)$
be an ordered family of convex bodies in~$\R^d$ each containing the origin and
let $\mu \in \R^d$ be a point.  The \Defn{Minkowski complex} associated to
$\P$ and $\mu$ is the simplicial complex $\Delta(\P;\mu)$ given by the
simplices $\sigma \subseteq [n]$ with
%\begin{equation}\label{eqn:gen_thres_weight}
\[
    \sigma \ \in \ \Delta(\P; \mu) \qquad \text{if and only if} \qquad \mu
    \ \notin \ P_\sigma \ := \ \sum_{i \in \sigma} P_i.
\]
%\end{equation}
Here, $\sum_{i \in \sigma}P_i \ = \ \{ \sum_{i \in \sigma} p_i : p_i \in P_i
\}$ is the  \Defn{Minkowski sum} (or vector sum) and we set $P_\emptyset :=
\{0\}$. By setting $P_i := \{ t \in \R : 0 \le t \le \l_i\}$, it follows that
threshold complexes are Minkowski complexes. For the case that each $P_i
\subset \R^d$ is an axis-parallel box, these simplicial complexes have
been studied by Pakianathan and Winfree~\cite{pakianathan2013} under the name
of \emph{quota complexes}. We may also replace a convex body $P_i$ by a
suitable convex polytope and we will tacitly do this henceforth.

\newcommand\CME{\mathrm{CM}E}%

Our motivation for studying Minkowski complexes comes from \emph{mixed Ehrhart
theory}. For a set $S \subset \R^d$, let us define the \Defn{discrete volume}
$E(S) := |S \cap \Z^d|$. The \Defn{discrete mixed volume} of lattice polytopes
$P_1,\dots,P_n \subset \R^d$ is defined as
\[
    \CME(P_1,\dots,P_n) \ = \ \sum_{I \subseteq [n]} (-1)^{n-|I|} E(P_I).
\]
It was shown in~\cite{JS} (see also~\cite{Bihan}) that, like its continuous
counterpart the \emph{mixed volume}, the discrete mixed volume satisfies 
\[
    0 \ \le \ \CME(Q_1,\dots,Q_n) \ \le \ \CME(P_1,\dots,P_n)
\] 
for lattice polytopes $Q_i \subseteq P_i$ for $i=1,\dots,n$. Since $\CME$ is
invariant under lattice translations, we may assume that $0 \in P_i$ for all
$i$. This allows us to express the discrete mixed volume as follows:
\begin{thm}\label{thm:CME}
    Let $\P = (P_1,\dots,P_n)$ be a family of $n > 0$ lattice polytopes in
    $\R^d$ with $0 \in P_i$ for all $i$. Then 
    \[
        (-1)^n \CME(P_1,\dots,P_n) \ = \ \sum_{\mu \in P_{[n]} \cap \Z^d}
        \tilde\chi(\Delta(\P;\mu)),
    \]
    where $\tilde\chi$ denotes the reduced Euler characteristic.
\end{thm}
\begin{proof}
    Since all polytopes $P_i$ contain the origin, it follows that $P_I
    \subseteq P_J$ for $I \subseteq J$.  Let us write $[P_I] : \R^d
    \rightarrow \{0,1\}$ for the characteristic function of $P_I$ and define
    $F := \sum_{I} (-1)^{n-|I|}[P_I]$. Then $\CME(P_1,\dots,P_n) = \sum_{\mu
    \in \Z^d} F(\mu)$. The result now follows from the observation that 
    \[
        (-1)^n F(\mu) \ = \ \sum_{\sigma \subseteq [n],\, \mu \notin P_\sigma}
        (-1)^{|\sigma|} \ = \ \tilde\chi(\Delta(\P,\mu))
    \]
    for all $\mu \in \R^d$.
\end{proof}

Theorem~\ref{thm:CME} prompted the question if Minkowski complexes have
restricted topology.  It was already noted in~\cite[Thm.~D.1]{pakianathan2013}
that this is not the case. We recall their result with a short proof.
\begin{prop}\label{prop:all_minkowski}
    For any simplicial complex $\Delta \subseteq 2^{[n]}$ there are polytopes
    $\P = (P_1,\dots,P_n)$ and $\mu$ in some $\R^D$ such that $\Delta =
    \Delta(\P;\mu)$.
\end{prop}
\begin{proof}
    Let $\sigma_1, \dots, \sigma_D$ be the facets of~$\Delta$. For $1 \le i
    \le n$ and $1 \le j \le D$ we set $t_{ij} := 1$ if $i \in \sigma_j$ and
    $t_{ij} := |\sigma_j|+1$ otherwise. For $1 \le i \le n$ define
    \[
        P_i \ := \ \{ p \in \R^D : 0 \le p_j \le t_{ij} \text{ for } 1 \le j
        \le D \}
    \]
    and let $\mu = (|\sigma_1|+1, \dots, |\sigma_d|+1)$.  Observe that $\mu
    \notin \sum_{i \in \sigma_j} P_i$ for all~$j$ and thus $\Delta \subseteq
    \Delta(\P; \mu)$. Conversely, if $\tau \not \in \Delta$, then for any $j$
    there is an $i$ with $i \in \tau \setminus \sigma_j$. This implies $\mu
    \in P_\tau$.
\end{proof}

\newcommand\ctd{\mathsf{ctd}}%

As a measure of complexity, define the \Defn{convex threshold dimension}
$\ctd(\Delta)$ of a simplicial complex $\Delta \subseteq 2^{[n]}$ as smallest
dimension $d$ in which $\Delta$ can be realized as a Minkowski complex. This
is exactly the minimum over $\dim P_1+\cdots+P_n$ over all $(\P,\mu)$ for
which $\Delta = \Delta(\P;\mu)$. Thus, the empty complex is the unique complex
of convex threshold dimension $0$ while $\ctd(\Delta) = 1$ if and only if
$\Delta$ is a threshold complex. Proposition~\ref{prop:all_minkowski} shows
that the convex threshold dimension is finite for every simplicial complex. In
the remainder we will show that $\ctd(\Delta)$ can be arbitrarily large and
hence is a \emph{proper} measure of the complexity of $\Delta$.

In~\cite{CH77} it was shown that threshold graphs are characterized by the
three forbidden induced subgraphs given in Figure~\ref{fig:forbidden}.
\begin{figure}[htpb]
    \includegraphics[width=.6\linewidth]{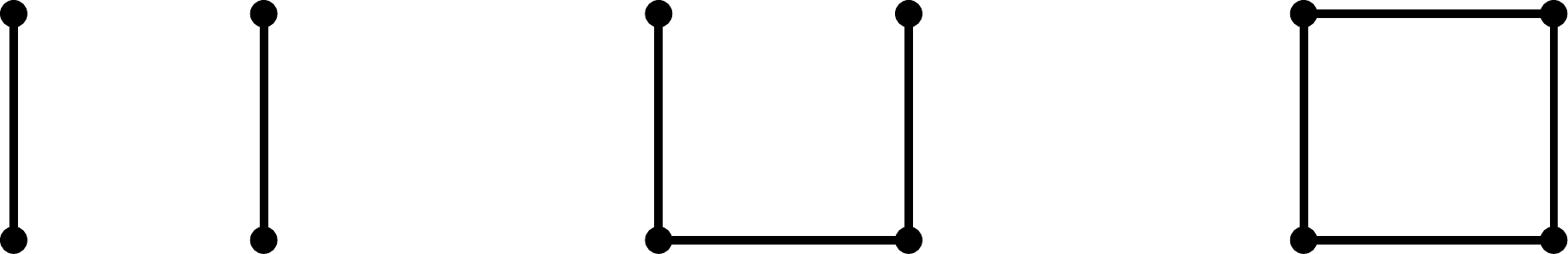}
    \caption{The three forbidden induced subgraphs for threshold graphs.}
    \label{fig:forbidden}
\end{figure}
The following result shows that these graphs are key to increasing the convex
threshold dimension. Let us denote by $\Delta * \Gamma = \{ \sigma \uplus \tau
: \sigma \in \Delta, \tau \in \Gamma \}$ the join of simplicial
complexes $\Delta$ and $\Gamma$.

\begin{thm}\label{thm:inc}
    Let $\Gamma$ be a simplicial complex that contains one of the graphs of
    Figure~\ref{fig:forbidden} as an induced subgraph. Then
    \[
        \ctd(\Delta * \Gamma) \ \ge \ \ctd(\Delta) + 1
    \]
    for every simplicial complex $\Delta$.
\end{thm}

\begin{cor}
    For every $d \ge 0$, there is a simplicial complex $\Delta$ with
    $\ctd(\Delta) \ge d$. 
\end{cor}

For the proof of Theorem~\ref{thm:inc}, we first note the following two
helpful facts about Minkowski complexes.

\begin{lem}\label{lem:lines-joins}
    Let $\Delta = \Delta(\P;\mu)$ be a Minkowski complex for $P_1,\dots,P_n
    \subset \R^d$. Let $\sigma, \tau \in \Delta$ be faces such that 
    $\sigma \cup \tau \in \Delta$. For any line $\ell$ through $\mu$, then the
    restrictions of $P_\sigma$ and $P_\tau$ to $\ell$ are on the same side
    of $\mu$.
\end{lem}

\begin{proof}
    Since $\sigma \cup \tau \in \Delta$ the convex set $P_{\sigma \cup \tau} =
    P_\sigma + P_\tau$ does not contain~$\mu$. By assumption on $\P$, we have
    $P_\sigma \cup P_\tau \subseteq P_{\sigma \cup \tau}$ from which the claim
    follows.
\end{proof}

\begin{lem}\label{lem:proj}
    Let $\Delta = \Delta(\P;\mu)$ be a Minkowski complex for a collection of
    polytopes $\P$ in $\R^d$. Suppose that there is a line $\ell$ through
    $\mu$ that does not intersect any~$P_\sigma$ for $\sigma \in \Delta$, then
    $\ctd(\Delta) < d$.
\end{lem}

\begin{proof}
    Denote by $\pi\colon \R^d \to \ell^\perp \cong \R^{d-1}$ the orthogonal
    projection along the line~$\ell$. Let $\pi(\P) = (\pi(P_1), \dots,
    \pi(P_n))$. For $\sigma \in \Delta$ the set $P_\sigma$ avoids~$\ell$ and
    thus $\pi(\mu) \notin \pi(P_\sigma)$. This shows that $\sigma \in
    \Delta(\pi(\P); \pi(\mu))$. Conversely $\sigma \in \Delta(\pi(\P);
    \pi(\mu))$ implies that the line $\ell = \pi^{-1}(\mu)$ does not
    intersect~$P_\sigma$ and thus $\sigma \in \Delta(\P; \mu)$. The claim now
    follows from  $\dim \pi(P_1)+\dots+\pi(P_n) < d$.
\end{proof}

\begin{proof}[Proof of Theorem~\ref{thm:inc}]
    Suppose that $\ctd(\Delta * \Gamma)=d$ and fix a realization $\Delta(\P;
    \mu)$ with $\P$ a collection of polytopes in~$\R^d$.  Denote the two edges
    of $\Gamma$ that induce one of the graphs in Figure~\ref{fig:forbidden} by
    $e$ and $f$ and let $e'$ and $f'$ be two disjoint (diagonal) nonedges.
    Then $\mu \in P_{e'} \cap P_{f'}$. Since $e \cup f = e' \cup f'$ we have
    $2\mu \in P_{e'} + P_{f'} = P_e + P_f$.  Thus, there is a vector $v \in
    \R^d\setminus\{0\}$ with $\mu - v \in P_e$ and $\mu+v \in P_f$. The
    line~$\ell$ 
    connecting $\mu-v$ and $\mu+v$ goes through~$\mu$. The convex sets $P_e$
    and $P_f$ intersect~$\ell$ on different sides of~$\mu$.
	
    The line~$\ell$ must also intersect $P_\sigma$ for some face $\sigma \in
    \Delta$ by Lemma~\ref{lem:proj}.  Since $\sigma \cup e, \sigma \cup f \in
    \Delta * \Gamma$, the sets $P_\sigma$ and $P_e$ as well as $P_\sigma$ and
    $P_f$ intersect~$\ell$ on the same side of~$\mu$ by
    Lemma~\ref{lem:lines-joins}. This is a contradiction.
\end{proof}

It would be very interesting if complexes of convex threshold dimension $d$
can be characterized in terms of the number of distinct copies of the
forbidden subgraphs.

\newcommand\X{\mathcal{X}}%
As a last thought, let us emphasize that convexity played a crucial role in
our considerations. For that, observe that the definition of Minkowski complex
$\Delta(\X;\mu)$ makes sense for collections $\X = (X_1,\dots,X_n)$ of
\emph{arbitrary} sets in $\R^d$ that contain the origin.

\begin{prop}\label{prop:any}
    Let $\Delta$ be a simplicial complex. Then there is a collection $\X
    =(X_1,\dots,X_n)$ of discrete sets in $\R$ such that $\Delta =
    \Delta(\X;\mu)$ for some $\mu \in \R$. In particular, there
    is a collection~$\X$ of contractible sets in $\R^2$ realizing $\Delta$ as
    a Minkowski complex.
\end{prop}
\begin{proof}
    Let $(\P;\mu)$ be a realization of $\Delta$ by convex sets in some $\R^d$.
    For any $\tau \not\in \Delta$, choose $y^\tau_i \in P_i$ such that $\mu =
    y^\tau_1 + \cdots + y^\tau_n$. Then $Y_i := \{0\} \cup \{ y^\tau_i : \tau
    \not \in \Delta \}$ yields a realization of $\Delta$ by discrete sets in
    $\R^d$.  The argument in the proof of Lemma~\ref{lem:proj} applies unless
    $d=1$ and proves the first claim. The second claim simply follows from the
    fact that we may connect the points $x_i \in X_i \setminus \{0\}$ by
    internally disjoint arcs properly contained in the upper half-plane to~$0$.
\end{proof}

\bibliographystyle{siam}

\begin{thebibliography}{1}

\bibitem{Bihan}
{\sc F.~Bihan}, {\em Irrational {M}ixed {D}ecomposition and {S}harp {F}ewnomial
  {B}ounds for {T}ropical {P}olynomial {S}ystems}, Discrete Comput. Geom., 55
  (2016), pp.~907--933.

\bibitem{CH77}
{\sc V.~Chv{\'a}tal and P.~L. Hammer}, {\em Aggregation of inequalities in
  integer programming}, in Studies in integer programming ({P}roc. {W}orkshop,
  {B}onn, 1975), North-Holland, Amsterdam, 1977, pp.~145--162. Ann. of Discrete
  Math., Vol. 1.

\bibitem{edelman}
{\sc P.~H. Edelman, T.~Gvozdeva, and A.~Slinko}, {\em Simplicial complexes
  obtained from qualitative probability orders}, SIAM J. Discrete Math., 27
  (2013), pp.~1820--1843.

\bibitem{Gol80}
{\sc M.~C. Golumbic}, {\em Algorithmic graph theory and perfect graphs},
  Academic Press [Harcourt Brace Jovanovich, Publishers], New
  York-London-Toronto, Ont., 1980.
\newblock With a foreword by Claude Berge, Computer Science and Applied
  Mathematics.

\bibitem{JS}
{\sc K.~Jochemko and R.~Sanyal}, {\em Combinatorial mixed valuations}.
\newblock Preprint, May 2016, 16 pages,
  \href{http://arxiv.org/abs/1605.07431}{arXiv:1605.07431}.

\bibitem{kalai}
{\sc G.~Kalai}, {\em Algebraic shifting}, in Computational commutative algebra
  and combinatorics ({O}saka, 1999), vol.~33 of Adv. Stud. Pure Math., Math.
  Soc. Japan, Tokyo, 2002, pp.~121--163.

\bibitem{Klivans}
{\sc C.~J. Klivans}, {\em Threshold graphs, shifted complexes, and graphical
  complexes}, Discrete Math., 307 (2007), pp.~2591--2597.

\bibitem{pakianathan2013}
{\sc J.~Pakianathan and T.~Winfree}, {\em Threshold complexes and connections
  to number theory}, Turkish J. Math., 37 (2013), pp.~511--539.

\bibitem{RRST}
{\sc J.~Reiterman, V.~R{\"o}dl, E.~{\v{S}}i{\v{n}}ajov{\'a}, and M.~Tuma}, {\em
  Threshold hypergraphs}, Discrete Math., 54 (1985), pp.~193--200.

\end{thebibliography}

\end{document}